\documentclass[11pt]{amsart}

\usepackage[a4paper,margin=1in]{geometry}
\usepackage{amsmath,amssymb,amsthm,mathrsfs,mathtools,physics}
\usepackage{enumerate}
\usepackage{hyperref}
\usepackage{tikz, amsfonts, color,  extarrows, cite, subfiles, csquotes, commath, eucal,mathrsfs, bbm }

\numberwithin{equation}{section}

\newtheorem{theorem}{Theorem}[section]
\newtheorem{proposition}[theorem]{Proposition}
\newtheorem{lemma}[theorem]{Lemma}

\theoremstyle{definition}
\newtheorem{definition}[theorem]{Definition}
\newtheorem{example}[theorem]{Example}
\newtheorem{remark}[theorem]{Remark}

\def\bc{\begin{center}}
\def\ec{\end{center}}

\def\beq{\begin{equation}}
\def\eeq{\end{equation}}
\def\beqarray{\begin{eqnarray*}}
\def\eeqarray{\end{eqnarray*}}
\def\<{\leftangle}
\def\>{\rightangle}
\def\({\left(}
\def\){\right)}

\def\<{\langle}
\def\>{\rangle}

\def\a{\alpha}

\def\O{\Omega}

\def\z{\zeta}

\def\w.r.t.{with respect to}

\def\N{{\mathbb{N}}}
\def\Z{{\mathbb{Z}}}
\def\D{{\mathbb{D}}}

\def\C{{\mathbb{C}}}

\def\bq{\begin{quote}}
\def\eq{\end{quote}}

\def\bit{\begin{itemize}}
\def\eit{\end{itemize}}

\def\ben{\begin{enumerate}}
\def\een{\end{enumerate}}

\hypersetup{
	colorlinks   = true,
	citecolor    = magenta}

\begin{document}

\title{Stability of the Monomial Basis Kernel of Reinhardt Domains}
\author{Shreedhar Bhat, Sahil Gehlawat}

\address{SB: Tata Institute of Fundamental Research
Centre for Applicable Mathematics,
Bengaluru 560065, India.}
\email{shreedhar24@tifrbng.res.in}

\address{SG: Department of Mathematics, Indian Institute of Technology Jodhpur, NH 65, Jodhpur, Rajasthan-342030, India}
\email{sahilg@iitj.ac.in}


\begin{abstract}
On a pseudoconvex Reinhardt domain $\Omega\subset\mathbb C^n$ the $p$-Bergman space $A^p(\Omega)$ admits a canonical basis of monomials indexed by a subset $S_p(\Omega)\subset\mathbb Z^n$. The corresponding $p$-Monomial Basis Kernel (or $p$-MBK) is defined by a series involving these monomials and their $L^p$–norms, and may be seen as an analogue of the Bergman kernel adapted to $L^p$–geometry. This article records stability properties of the $p$-MBK and of the index set $S_p(\Omega)$ with respect to the parameter $p$. First, under mild hypotheses, the $p$-MBK depends continuously on $p\in[1,\infty)$, and a Ramadanov-type theorem holds for $p$-MBK for an increasing sequence of pseudoconvex Reinhardt domains $\{\Omega_{j}\}_{j \ge 1}$. Second, for certain special classes of monomial polyhedra, we explicitly compute $S_p(\Omega)$ and the associated Threshold exponents. Finally, these explicit models are used to illustrate structural properties of $S_p(\Omega)$ under finite unions, intersections, and products.
\end{abstract}

\maketitle

\section{Introduction}
For a domain $\Omega \subset \mathbb{C}^n$, the Bergman space corresponding to $\Omega$ is the space of $L^2$-integrable holomorphic functions on $\Omega$, that is
\[
A^{2}(\Omega) = \left\{f \in \mathcal{O}(\Omega) \mid \int_{\Omega} \vert f\vert^2 dV < +\infty\right\}.
\]
The study of kernel functions is central to complex analysis and among these, the Bergman kernel plays a distinguished role. The Bergman kernel $K_{\Omega}(\cdot,\cdot)$ of a domain $\Omega \subset \mathbb{C}^n$ is the reproducing kernel of the Bergman space and it encodes the $L^2$–geometry of holomorphic functions. This gives rise to the Bergman projection, which is the orthogonal projection $P : L^2(\Omega) \to A^2(\Omega)$ given by
\[
P(f)(z) = \int_{\Omega}f(w) K_{\Omega}(z,w) dV(w),
\]
where $dV$ is the Lebesgue measure.


But sometimes the $L^2$–framework is too rigid to capture fine $L^p$–phenomena for $p\neq 2$, especially on non-smooth domains (for example see, \cite[Theorem 1.1]{Deng2020LinearVariations}). This motivates an extensive study of the $p$-Bergman spaces,
\[
A^{p}(\Omega) = \left\{f \in \mathcal{O}(\Omega) \mid \int_{\Omega} \vert f\vert^p dV < \infty\right\},
\]
and the associated Bergman projection on $L^p$–spaces for $1 \le p < \infty$, with a particular focus on their duality and regularity properties, as well as on their dependence on the geometry of the domain (see \cite{zeytuncu2020survey}, \cite{Hedenmalm2002TheDomains}, \cite{chakrabarti2016Lp}, \cite{chen2020lp},  \cite{chakrabarti2019duality},\cite{bhat2024duality}, \cite{Chen2021OnTheory}).

For a Reinhardt domain $\Omega \subset \mathbb{C}^n$, it is well known that the Bergman kernel $K_{\Omega}$ can be described as a sum over a suitable set of monomials, as they form an orthogonal basis for the corresponding Bergman space $A^2(\Omega)$ (cf \cite{Jarnicki2017FirstDomains}). With this as motivation, Chakrabarti and Edholm in \cite{chakrabarti2024projections} introduced the notion of the $p$-Monomial Basis Kernel $K_{p,\Omega}$ ($p$-MBK for short), for the space $A^p(\Omega)$ for a pseudoconvex Reinhardt domain $\Omega$. To recall, the $p$-Monomial Basis Kernel is built from monomials in the space $A^p(\Omega)$, weighted by their $L^p$–norms (see Definition \ref{D:p-MBK}). For $p=2$, this coincides with the classical Bergman kernel $K_{\Omega}$. Using the $p$-Monomial Basis Kernel, they \emph{formally} defined the monomial basis projection $M_{p} : L^{p}(\Omega) \to A^{p}(\Omega)$.
{For certain classes of domains, the monomial basis projection exhibits superior mapping properties on $L^p$ spaces compared to the Bergman projection. For example, when $U$ is a monomial polyhedron, both the $p$-Monomial Basis Projection ($p$-MBP) and its associated absolute-value operator are bounded on $L^p(U)$ for all relevant exponents $p$. (See \cite[\S 8]{chakrabarti2024projections}.) By contrast, for such domains the Bergman projection fails to be $L^p$-bounded once $p$ leaves the critical interval determined by the threshold exponents associated with $U$. The persistence of $L^p$-boundedness for the $p$-MBP and its absolute operator beyond this critical range highlights a genuinely stronger regularity theory than that available for the Bergman projection. Consequently, the monomial basis projection emerges as a natural and intrinsically interesting object for further investigation in the study of function spaces on monomial polyhedra.}\\

In \cite[page 41]{chakrabarti2024projections}, the authors noted that the study of the $p$-allowable indices set $S_{p}(\Omega)$ (Definition \ref{D:p-allowable indices}) is central in the $L^p$-function theory of a Reinhardt domain $\Omega \subset \mathbb{C}^n$. These are the set of indices for which the corresponding monomials forms a basis of the $p$-Bergman space $A^{p}(\Omega)$. One of the important questions in the study of the set $S_{p}(\Omega)$ is the dependence on $p$. For a Reinhardt domain, we say $p \in (1, \infty)$ is a \textit{Threshold exponent} if for any $\epsilon >0$, there exists $q \in (p-\epsilon, p+ \epsilon)$ such that $S_{p}(\Omega) \neq S_{q}(\Omega)$ (see Definition \ref{D:Threshold}). In a sense, Threshold exponents are the points of discontinuity of the set of $p$-allowable indices $S_{p}(\Omega)$. The Threshold exponent are important in understanding the $L^{p}$-irregularity of Bergman projection. For instance, in \cite{bender2022regularity}, it was proved that for a monomial polyhedron $\Omega$, the Bergman projection is bounded in $L^{p}$ if and only if $p \in (\tilde{q}, \tilde{p})$, where $\tilde{p}$ is the smallest Threshold exponent bigger than 2, and $\tilde{q}$ is the conjugate of $\tilde{p}$.

The purpose of this article is two-fold. First, we study the stability of the $p$-Monomial Basis Kernel for a pseudoconvex Reinhardt domain. More specifically, we prove that the $p$-Monomial Basis Kernel depends continuously on $p$ (Theorem \ref{thm:continuity-p}). We also prove a Ramadanov type result for $p$-Monomial Basis Kernel corresponding to an increasing sequence of pseudoconvex Reinhardt domains. 



In the second part, we turn our attention to monomial polyhedra, a class of Reinhardt domains defined by finitely many inequalities involving monomials. Recall, the monomial polyhedra corresponding to a matrix $A=(a_{ij})\in GL_n(\mathbb Z)$ is defined by
\[
\Omega_A:=\left\{z\in\mathbb C^n : \bigl|\prod_{i=1}^n z_i^{a_{ij}}\bigr|<1,\ \text{for }j=1,\dots,n\right\}.
\]
We provide explicit formulae for the $p$–allowable index sets and the corresponding Threshold exponents for some special families of monomial polyhedra. This includes Hartogs triangles $H_{\gamma}$, the two-dimensional monomial polyhedra $\Omega_{A}$ associated with the $2\times 2$ matrix
\[
A=\begin{pmatrix} a & -b\\ -c & d\end{pmatrix},
\]
where $a,b,c,d >0$, and det$(A) > 0$. We also consider ``type 1'' and ``type 2'' monomial polyhedra $\Omega_{B}, \Omega_{C} \subset \mathbb{C}^n$ associated to the $n \times n$ matrices 
\[
B=\begin{pmatrix}
k_1 & -k_2 & \cdots & -k_n\\
0 & 1 & & 0\\
\vdots & & \ddots & \vdots\\
0 & 0 & \cdots & 1
\end{pmatrix} \ \quad  \& \quad
 \ C=\begin{pmatrix}
k_1 & -k_2 & 0 & \cdots & 0\\
0 & k_2 & -k_3 & \cdots & 0\\
\vdots & & \ddots & \ddots & \vdots\\
0 & \cdots & 0 & k_{n-1} & -k_n\\
0 & \cdots & 0 & 0 & k_n
\end{pmatrix},
\] 
where $k_{i} \in \mathbb{N}$, for all $1 \le i \le n$, and gcd$(k_{1}, k_{2}, \ldots ,k_{n}) = 1$. These domains have already been considered before, especially with regard to the calculation of the Bergman kernel and the question of the boundedness of the Bergman projection (see \cite{bender2022regularity}, \cite{Chakrabarti2020elementary}, \cite{Chakrabarti2024kernel}, \cite{chen2026lower}, \cite{Zha21Hartogs} and \cite{Zha21Reinhardt}). For Hartogs triangle $H_{\gamma}$ with irrational exponent $\gamma$, we show that every $p\in[1,\infty)$ is a Threshold exponent.

Finally, we collect several structural properties of $S_p(\Omega)$ under finite unions, intersections, and products of Reinhardt domains. The case of finite unions behaves especially well, with $S_p(\Omega_1\cup\Omega_2)=S_p(\Omega_1)\cap S_p(\Omega_2)$, where $\Omega_{1}, \Omega_{2} \subset \mathbb{C}^n$ are Reinhardt domains. Meanwhile, intersections and infinite unions display more subtle behaviour. We show that $S_p(\Omega)$ is determined by the geometry of the domain near the origin and at infinity (see Theorem \ref{thm:localization}).

The paper is organized as follows. In Section~\ref{sec:prelim}, we review basic facts about Reinhardt domains, $p$-Bergman spaces, and $p$–allowable index sets, and we introduce the $p$-monomial basis kernel. Section~\ref{sec:continuity} is devoted to continuity of the $p$-MBK with respect to $p$ and to a local uniform domination result for the summands. In Section~\ref{sec:polyhedra}, we carry out explicit computations of $S_p(\Omega)$ and Threshold exponents for special classes of monomial polyhedra, that is, $\Omega_{A}, \Omega_{B}$, and $ \Omega_{C}$. Section~\ref{sec:Sp-properties} then records structural properties of $S_p(\Omega)$ under set-theoretic operations and proves a Ramadanov-type convergence theorem.

\section{Preliminaries on Reinhardt domains and \texorpdfstring{$p$}{p}-Bergman spaces}
\label{sec:prelim}

\noindent In this section, we recall the basic notions used throughout the paper and fix notations.

\subsection{Reinhardt domains and their shadows}

\begin{definition}
A domain (i.e. a non-empty open connected set) $\Omega\subset\mathbb C^n$ is called a \emph{Reinhardt domain} if
\[
(z_1,\dots,z_n)\in\Omega \Longrightarrow (e^{i\tau_1}z_1,\dots,e^{i\tau_n}z_n)\in\Omega
\]
for all $(\tau_1,\dots,\tau_n)\in\mathbb R^n$.
\end{definition}

Associated to a Reinhardt domain $\Omega$ is its radial image in $\mathbb R^n$.

\begin{definition}
The \emph{Reinhardt shadow} of a Reinhardt domain $\Omega\subset\mathbb C^n$ is the set
\[
|\Omega|:=\{(|z_1|,\dots,|z_n|)\in\mathbb R^n: (z_1,\dots,z_n)\in\Omega\}.
\]
We also write
\[
|\Omega|\times\mathbb T^n:=\{(r_1e^{i\tau_1},\dots,r_ne^{i\tau_n})\in\mathbb C^n : (r_1,\dots,r_n)\in|\Omega|,\ \tau_j\in\mathbb R, j=1,\cdots,n\}
\]
where $\mathbb T^n$ is the $n$-torus.
\end{definition}

For $q\ge 1$ we will also need a mixed geometric mean of two copies of $|\Omega|$.

\begin{definition}
For a subset $E\subset[0,\infty)^n$ and $q\ge 1$ we define
\[
E^{1/q}\cdot E^{(q-1)/q}:=\{(s_1^{1/q}\cdot t_1^{(q-1)/q},\dots,s_n^{1/q}\cdot t_n^{(q-1)/q}) : s,t\in E\}.
\]
\end{definition}

\noindent Recall that when $\Omega$ is pseudoconvex and Reinhardt, its shadow is log-convex, and this description simplifies.

\begin{remark} 
If $\Omega\subset\mathbb C^n$ is a pseudoconvex Reinhardt domain, then $|\Omega|$ is log-convex. In particular, for every $q\ge 1$,
\[
|\Omega|^{1/q}\cdot|\Omega|^{(q-1)/q}\times\mathbb T^n=\Omega.
\]
\end{remark}

\begin{definition}\cite[Section 3.1]{chakrabarti2024projections}
     For $p\ge 1$, a ``twisting" map $\chi_p$ is defined on $\mathbb{C}^n$ by
\[
\chi_p(\z):=\bigl(\z_1\abs{\z_1}^{(p-2)},\dots,\z_n\abs{\z_n}^{(p-2)}\bigr).
    \] 
 The map $\chi_p$ is a homeomorphism onto itself, and its inverse is given by $\chi_q$ where $1/p+1/q=1$. Also, when $ p = 2$, the twisting map $\chi_2$ reduces to an identity map.
\end{definition}

\subsection{\texorpdfstring{$p$}{p}-Bergman spaces and \texorpdfstring{$p$}{p}–allowable indices}

Let $\Omega\subset\mathbb C^n$ be a domain and $1\le p<\infty$. The $p$-Bergman space is the closed subspace of $L^p(\Omega)$ consisting of holomorphic functions.


\noindent When $\Omega$ is a Reinhardt domain, monomials are natural building blocks.

\begin{definition}\label{D:p-allowable indices}
For $\alpha=(\alpha_1,\dots,\alpha_n)\in\mathbb Z^n$ and $z=(z_1,\dots,z_n)\in\mathbb C^n$ we set
\[
e_\alpha(z):=z_1^{\alpha_1}\cdots z_n^{\alpha_n},
\]
with the convention that negative exponents are interpreted as Laurent monomials whenever they are locally integrable. The \emph{$p$–allowable index set} of $\Omega$ is
\[
S_p(\Omega):=\{\alpha\in\mathbb Z^n : e_\alpha\in A^p(\Omega)\}.
\]
\end{definition}

\begin{remark}
By Hölder's inequality, these index sets form an inclusion chain in $p$. In particular, if $1\le p_1\leq p_2\leq p_3<\infty$ and $\Omega\subset\mathbb C^n$ is a domain, then
\[
S_{p_2}(\Omega)\supseteq S_{p_1}(\Omega)\cap S_{p_3}(\Omega).
\]
\end{remark}

The dependence of $S_p(\Omega)$ on $p$ is typically piecewise constant, with jumps at certain critical values.

\begin{definition}\label{D:Threshold}
A number $p\in[1,\infty)$ is called a \emph{Threshold exponent} for a domain $\Omega$ if for every $\varepsilon>0$ there exists $q\in(p-\varepsilon,p+\varepsilon)$ such that $S_q(\Omega)\neq S_p(\Omega)$.
\end{definition}

Threshold exponents record the values of $p$ at which the integrability of some monomial $e_\alpha$ changes, and carry important geometric information about the domain. In particular, they are closely related to indices of duality, regularity, and integrability for $A^p(\Omega)$. See, for example, \cite[Theorem 4.2, Proposition 4.8]{chakrabarti2019duality} or \cite{bhat2024duality}. Also, for monomial polyhedra they can be computed explicitly in many cases. 

\subsection{The \texorpdfstring{$p$}{p}-monomial basis kernel}

Let $\Omega\subset\mathbb C^n$ be a Reinhardt domain and $1\le p<\infty$. We now define the $p$-monomial basis kernel ($p$-MBK) associated to the $p$-allowable monomials.

\begin{definition}\label{D:p-MBK} \cite[Section 3]{chakrabarti2024projections}
For $z,w\in\Omega$ and $\alpha\in\mathbb Z^n$ set
\[
E_{\alpha,p}^\Omega(z,w):=
\begin{cases}
\dfrac{e_\alpha(z)\overline{e_\alpha(w)}|e_\alpha(w)|^{p-2}}{\|e_\alpha\|_{p,\Omega}^p\,}, & \alpha\in S_p(\Omega),\\[5pt]
0, & \alpha\notin S_p(\Omega),
\end{cases}
\]
where and $\|\cdot\|_{p,\Omega}$ the $L^p$–norm. The \emph{$p$-monomial basis kernel} on $\Omega$ is
\beq\label{def:p-MBK}
K_p^\Omega(z,w):=\sum_{\alpha\in\mathbb Z^n}E_{\alpha,p}^\Omega(z,w)=\sum_{\alpha\in S_p(\Omega)}\dfrac{e_\alpha(z)\overline{e_\alpha(w)}|e_\alpha(w)|^{p-2}}{\|e_\alpha\|_{p,\Omega}^p\,}.
\eeq
When the domain is clear from the context, we write $\|\cdot\|_{p}$ , $E_{\alpha,p}$ and $K_p$ instead of  $\|\cdot\|_{p,\Omega}$, $E_{\alpha,p}^\Omega$ and $K_p^\Omega$.
\end{definition}

{\cite{chakrabarti2024projections} showed that when $\O$ is a pseudoconvex Reinhardt domain, the summation in \ref{def:p-MBK} converges uniformly on compact subsets of $\Omega\times \Omega$}. Formally, when $p=2$ and the monomials $\{e_\alpha:\alpha\in S_2(\Omega)\}$ form an orthogonal basis of $A^2(\Omega)$, the kernel $K_2$ coincides with the usual Bergman kernel. 

\section{Continuity of the \texorpdfstring{$p$}{p}-monomial basis kernel in \texorpdfstring{$p$}{p}}
\label{sec:continuity}

In this section, we investigate the stability of the $p$-MBK with respect to the parameter $p$, and establish that on pseudoconvex Reinhardt domains the $p$-MBK depends continuously on the exponent $p$. The basic idea is to show pointwise convergence of each summand $E_{\alpha,p}$ and to dominate the family by an $\ell^1$–summable majorant independent of $p$ in a compact interval.





First, we recall this result from \cite[Lemma 1.6.2]{Jarnicki2017FirstDomains}. 
\begin{remark}\label{rem:CompactSet} Assume that $K$ is a compact Reinhardt subset of $\Omega$. Then there exists a compact set $K'\subset \Omega$ and $\theta \in(0,1)$  such that any $\alpha \in \mathbb{Z}^n$ that satisfies the condition 
\[ \max_{z \in K} |z^\alpha| \le \theta^{|\alpha|} \max_{z \in K'} |z^\alpha|, \]
where the norm $|\alpha|$ is defined as the sum of its absolute components, $|\alpha_1| + \dots + |\alpha_n|$.
\end{remark}

We show that for each fixed monomial index $\alpha$, the corresponding summand $E_{\alpha,p}$ varies continuously with $p$ along monotone sequences, provided the index remains $p$–allowable.

\begin{lemma}\label{lem:pointwise}
Let $\Omega\subset\mathbb C^n$ be a Reinhardt domain, $\alpha\in\mathbb Z^n$, and $p\in[1,\infty)$. Let $\{p_k\}_{k=1}^\infty$ be a monotone sequence (either increasing or decreasing) with $p_k\to p$ and assume that $\alpha\in S_{p_k}(\Omega)$ for all $k$. Then
\[
E_{\alpha,p_k}(z,w)\longrightarrow E_{\alpha,p}(z,w)
\]
for every $z,w\in\Omega$.
\end{lemma}

\begin{proof}
There are two cases.

\smallskip\noindent\textit{Case 1: $\alpha\in S_p(\Omega)$.}

In this case $e_\alpha\in L^p(\Omega)$, and by monotone convergence and basic properties of $L^q$–norms, $\|e_\alpha\|_{L^{p_k}(\Omega)}\to\|e_\alpha\|_{L^p(\Omega)}$ as $k\to\infty$. Furthermore, the factor $\|e_\alpha(w)\|^{p_k-2}$ converges to $\|e_\alpha(w)\|^{p-2}$ for each fixed $w$, since it is a continuous function of the exponent. Combining these facts yields convergence of the quotient defining $E_{\alpha,p_k}(z,w)$ to $E_{\alpha,p}(z,w)$.

\smallskip\noindent\textit{Case 2: $\alpha\notin S_p(\Omega)$.}

Here $e_\alpha\notin L^p(\Omega)$ while $e_\alpha\in L^{p_k}(\Omega)$ for all $k$. Suppose by contradiction that $\|e_\alpha\|_{L^{p_k}(\Omega)}^{p_k}$ remains bounded by some constant $M>0$ along the sequence. Then by Fatou's lemma we would have
\[
\int_\Omega |e_\alpha|^p \le \liminf_{k\to\infty}\int_\Omega |e_\alpha|^{p_k}\le M,
\]
which contradicts the assumption that $e_\alpha\notin L^p(\Omega)$. Hence $\|e_\alpha\|_{L^{p_k}(\Omega)}^{p_k}\to\infty$. Consequently $E_{\alpha,p_k}(z,w)\to 0$ for all $z,w\in\Omega$, which agrees with $E_{\alpha,p}(z,w)$ by definition.
\end{proof}


To pass from pointwise convergence of summands to continuity of the full kernel, we need a uniform $\ell^1$–majorant for the family $\{E_{\alpha,q}\}$ when $(z,w)$ is restricted to a compact set and $q$ ranges in a neighbourhood of a fixed exponent $p$. This is provided by a Bergman-type inequality.

\begin{lemma}\label{lem:domination}
Let $\Omega\subset\mathbb C^n$ be a pseudoconvex Reinhardt domain, $p\ge 1$, and $K\Subset\Omega$ a compact set. Then there exists a non-negative function $g_{K,p}:\mathbb Z^n\to[0,\infty)$ such that
\[
|E_{\alpha,q}(z,w)|\le g_{K,p}(\alpha)
\]
for all $z,w\in K$, all $q$ in a neighbourhood of $p$, and all $\alpha\in\mathbb Z^n$. Moreover,
\[
\sum_{\alpha\in S_p(\Omega)\cup S_q(\Omega)} g_{K,p}(\alpha)<\infty
\]
for every such $q$.
\end{lemma}

\begin{proof}
Before proving the lemma, we first recall the Bergman inequality for functions in $A^{p}$. 
Let $f \in A^{p}(\Omega)$ and let $z \in K \subset\subset \Omega$. Then
\[
|f(z)|^{p} 
\le \frac{1}{\mathrm{vol}(B(0,r))}\int_{B(z,r/2)} |f(\zeta)|^{p}\, dV(\zeta)
\le \frac{1}{\mathrm{vol}(B(0,r))}\int_{\Omega} |f(\zeta)|^{p}\, dV(\zeta),
\]
where $r=\tfrac12 \mathrm{dist}(K,\partial\Omega)$. 
Consequently,
\[
|f(z)|^{p} \le C_{K}\,\|f\|_{p}^{p},
\]
where the constant $C_{K}$ is independent of both $p$ and $z$.

Let $K$ be a compact subset of $\Omega$ and let $z,w\in K$. 
Observe that $\chi_{q}$ converges uniformly to $\chi_{p}$ on compact subsets as $q\to p$. 
Hence there exists a compact set $K'$ such that
\[
|K|^{\frac1q}\,\cdot |K|^{\frac1q-1}\times \mathbb{T}^{n} \subset K' \subset \Omega
\]
for all $q$ in a neighborhood of $p$. \\
Further using remark \ref{rem:CompactSet}, there exist a compact set $K'' \subset \Omega$, $\theta \in (0,1)$ (corresponding to $K'$). Also, there exists a point 
$
t_{q}=(t_{1,q},t_{2,q},\ldots,t_{n,q})\in K'\subset K''
$ such that $t^{q}_{j,q}=z_{j}\,\overline{w}_{j}\,|w_{j}|^{\,q-2}$.\\
Finally it follows that
\[
\bigl|E_{\alpha,q}(z,w)\bigr|
=
\left|
\frac{e_{\alpha}(z)e_{\alpha}(w)\,|e_{\alpha}(w)|^{\,q-2}}
{\|e_{\alpha}\|_{q}^{q}}
\right|
=
\frac{|e_{\alpha}(t_{q})|^{q}}{\|e_{\alpha}\|_{q}^{q}}
\le \theta^{\,|\alpha|}\,C_{K''}.
\]
Therefore, as in Remark~3.1, we set
\[
g_{\alpha}(z)=\theta^{\,|\alpha|}\,C_{K''},
\]
where $\theta\in(0,1)$ depends on $K'$. 
This shows that
\[
\sum_{\alpha\in\mathbb{Z}^{n}} g_{K,p}(\alpha)
= C_{K''}\sum_{\alpha\in\mathbb{Z}^{n}} \theta^{\,|\alpha|}
<\infty.
\]
\end{proof}


\noindent We can now state and prove the main continuity result.

\begin{theorem}\label{thm:continuity-p}
Let $\Omega\subset\mathbb C^n$ be a pseudoconvex Reinhardt domain and $p\in[1,\infty)$. Then the $p$-monomial basis kernel $K_p(z,w)$ depends continuously on $p$ in the sense that
\[
\lim_{q\to p}K_q(z,w)=K_p(z,w)
\]
for every $z,w\in\Omega$. Moreover, the convergence is locally uniform on $\Omega\times\Omega$.
\end{theorem}

\begin{proof}
Fix $p\in[1,\infty)$ and $z,w\in K \subset \subset \Omega$. For $q$ close to $p$ we write
\begin{multline*}
K_p(z,w)-K_q(z,w)=\sum_{\alpha\in S_p(\Omega)\cap S_q(\Omega)}\bigl(E_{\alpha,p}(z,w)-E_{\alpha,q}(z,w)\bigr)-\sum_{\alpha\in S_q(\Omega)\setminus S_p(\Omega)}E_{\alpha,q}(z,w)\\ +\sum_{\alpha\in S_p(\Omega)\setminus S_q(\Omega)}E_{\alpha,p}(z,w).
\end{multline*}
By Lemma~\ref{lem:pointwise}, for each fixed $\alpha$ the functions $E_{\alpha,q}(z,w)$ converge to $E_{\alpha,p}(z,w)$ as $q\to p$ whenever $\alpha$ remains allowable, and converge to $0$ otherwise. By Lemma~\ref{lem:domination}, for $q$ in a small neighbourhood of $p$ there exists a summable function $g_{K,p}(\alpha)$ dominating $|E_{\alpha,q}(z,w)|$ uniformly in $q$. Hence the family of partial sums satisfies the hypotheses of the dominated convergence theorem on the discrete index set $\mathbb Z^n$, and we conclude that
\[
\lim_{q\to p}\bigl(K_p(z,w)-K_q(z,w)\bigr)=0.
\]
\end{proof}

\section{\texorpdfstring{$p$}{p}-Bergman spaces and Threshold exponents for monomial polyhedra}
\label{sec:polyhedra}

We now focus our attention on certain special classes of monomial polyhedra. These examples of monomial polyhedra were noted in \cite[section 1.6]{bender2022regularity}.

\begin{definition}
    A bounded domain ${U} \subset \mathbb{C}^n$ is called a \emph{monomial polyhedron} if there exist precisely $n$ monomials $e_{\alpha_1},\ldots,e_{\alpha_n}$ such that
\[
{U}= \{\, z \in \mathbb{C}^n : |e_{\alpha_1}(z)| < 1,\; \ldots,\; |e_{\alpha_n}(z)| < 1 \,\}.
\]
\end{definition}

\subsection{Monomial polyhedra associated to integer matrices}

Let $A=(a_{ij})\in GL_n(\mathbb Z)$ be an invertible integer matrix. To $A$ we associate the monomial polyhedra
\[
\Omega_A:=\left\{z\in\mathbb C^n : \bigl|\prod_{i=1}^n z_i^{a_{ij}}\bigr|<1,\ \text{for }j=1,\dots,n\right\}.
\]
These are bounded pseudoconvex Reinhardt domains whose shadows are defined by linear inequalities in logarithmic coordinates. For concreteness, we begin with the case $n=2$.

\begin{example}\label{ex:OmegaA-2D}
Let
\[
A=\begin{pmatrix} a & -b\\ -c & d\end{pmatrix},
\]
with $a,b,c,d>0$ and $\det A=ad-bc>0$, and assume $\gcd(a,b)=\gcd(c,d)=1$. Then
\[
\Omega_A=\{(z_1,z_2)\in\mathbb C^2: |z_1|^a<|z_2|^b,\ |z_2|^d<|z_1|^c\}.
\]
\end{example}

\noindent We now compute the $p$–allowable index set for $\Omega_A$.

\begin{proposition}\label{prop:Sp-OmegaA}
For the domain $\Omega_A$ of Example~\ref{ex:OmegaA-2D} and $1\le p<\infty$, the $p$–allowable index set is
\[
S_p(\Omega_A)=\left\{(\alpha_1,\alpha_2)\in\mathbb Z^2 : (a\alpha_2+b\alpha_1)p>-2(a+b),\ (c\alpha_2+d\alpha_1)p>-2(c+d)\right\}.
\]
Moreover, the set of Threshold exponents is
\[
\mathrm{Threshold}(\Omega_A)=\left\{\frac{2(a+b)}{j},\ \frac{2(c+d)}{k} : 1\le j\le 2(a+b)-1,\ 1\le k\le 2(c+d)-1\right\}.
\]
\end{proposition}

\begin{proof}
Suppose $(\alpha_{1}, \alpha_{2}) \in S_{p}(\Omega_{A})$ and $p \in [1, +\infty)$. Consider
\begin{equation}\label{E:E1}
\norm{e_{\alpha}}^{p}_{p,\Omega_{A}} = \int_{\Omega_{A}}{\vert e_{\alpha}\vert^{p} dV} = \int_{\Omega_{A}}{\vert z_{1}\vert^{\alpha_{1} p} \vert z_{2}\vert^{\alpha_{2} p} dV}.
\end{equation}
If $(z_{1},z_{2}) \in \Omega_{A}$, then $0 <\vert z_{2}\vert < 1$, and $\vert z_{2}\vert^{\frac{d}{c}} < \vert z_{1}\vert < \vert z_{2}\vert^{\frac{b}{a}}$.
Putting this in Equation \ref{E:E1}, we get
\begin{align*}
\norm{e_{\alpha}}^{p}_{p,\Omega_{A}} &= (2\pi)^{2}\int_{0}^{1} \int_{r_{2}^{\frac{d}{c}}}^{r_{2}^{\frac{b}{a}}} r_{1}^{\alpha_{1} p +1} r_{2}^{\alpha_{2} p +1} dr_{1} dr_{2}\\
&= \frac{(2 \pi)^2}{(\alpha_{1}p + 2)} \int_{0}^{1}(r_{2}^{\frac{b}{a}(\alpha_{1}p+2)}-r_{2}^{\frac{d}{c}(\alpha_{1}p+2)})r_{2}^{\alpha_{2}p +1} dr_{2}\\
&= \frac{(2 \pi)^2}{(\alpha_{1}p +2)} \left[\int_{0}^{1}r_{2}^{\frac{(b \alpha_{1} + a \alpha_{2})p + 2b + a}{a}} dr_{2} - \int_{0}^{1}r_{2}^{\frac{(d \alpha_{1} + c \alpha_{2})p + 2d + c}{c}} dr_{2}\right]\\
&= \frac{(2 \pi)^2}{(\alpha_{1}p + 2)} \left[\frac{a}{(b \alpha_{1} + a \alpha_{2})p + 2(b+a)} - \frac{c}{(d \alpha_{1} + c \alpha_{2})p + 2(d+c)}\right]\\
&= \frac{(2 \pi)^{2} (ad-bc)}{((b \alpha_{1} + a \alpha_{2})p + 2(b+a)) ((d \alpha_{1} + c \alpha_{2})p + 2(d+c))}.
\end{align*}

\noindent Therefore, $\norm{e_{\alpha}}^{p}_{p,\Omega_{A}} < +\infty$ if and only if
\[
(b \alpha_{1} + a \alpha_{2})p + 2(b+a) > 0 \ \  \& \ \ (d \alpha_{1} + c \alpha_{2})p + 2(d+c) > 0.
\]
It is clear that the Threshold exponents are the values of $p \in [1, +\infty)$ such that atleast one of the above inequalities becomes an equality of some index $(\alpha_{1}, \alpha_{2}) \in \mathbb{Z}^2$. Define
\begin{align*}
&S_{1} := \{p \in [1, +\infty) : p = \frac{-2(a+b)}{\alpha_{2}a + \alpha_{1} b}, \ \text{for some} \ (\alpha_{1}, \alpha_{2}) \in \mathbb{Z}^2\},\\
&S_{2} := \{p \in [1, +\infty) : p = \frac{-2(c+d)}{\alpha_{2}c + \alpha_{1} d}, \ \text{for some} \ (\alpha_{1}, \alpha_{2}) \in \mathbb{Z}^2\}.
\end{align*}
Suppose $p \in S_{1}$. The condition $p \ge 1$ gives $0 < -(\alpha_{2}a + \alpha_{1} b) \le 2(a+b)$. Since $\text{g.c.d}(a,b) = 1$, therefore $-(\alpha_{2}a + \alpha_{1} b) \in \{l \in \mathbb{Z}: \ 1 \le l \le 2(a+b) -1\}.$
Thus, we have
\[
S_{1} = \left\{\frac{2(a+b)}{l} : l \in \mathbb{Z}, \ 1 \le l \le 2(a+b) - 1\right\}.
\]
Similarly,
\[
S_{2} = \left\{\frac{2(c+d)}{l} : l \in \mathbb{Z}, \ 1 \le l \le 2(c+d) - 1\right\}.
\]
From above, we know $\mathrm{Threshold}(\O_A) = S_{1} \cup S_{2}$, therefore
\[
\mathrm{Threshold}(\O_A) = \left\{\frac{2(a+b)}{j}, \frac{2(c+d)}{k} : j, k \in \mathbb{Z}, \ 1 \le j \le 2(a+b) - 1, \ \& \
1 \le k \le 2(c+d) - 1\right\}.
\]
\end{proof}


\subsection{Higher dimensional monomial polyhedra of ``type 1'' and ``type 2''}

We now consider two families of higher dimensional monomial polyhedra.

\begin{example}[Type 1]\label{ex:type1}
Let $k_1,\dots,k_n\in\mathbb N$ with $\gcd(k_1,\dots,k_n)=1$, and consider the matrix
\[
B=\begin{pmatrix}
k_1 & -k_2 & \cdots & -k_n\\
0 & 1 & & 0\\
\vdots & & \ddots & \vdots\\
0 & 0 & \cdots & 1
\end{pmatrix}.
\]
The associated domain is
\[
\Omega_B=\left\{(z_1,\dots,z_n)\in\mathbb C^n : |z_1|^{k_1}<|z_2|^{k_2}\cdots|z_n|^{k_n},\ |z_j|<1,\ j=2,\dots,n\right\}.
\]
\end{example}

\begin{proposition}\label{prop:Sp-type1}
Let $\Omega_B$ be as in Example~\ref{ex:type1} and $1\le p<\infty$. Then
\[
S_p(\Omega_B)=\left\{(\alpha_1,\dots,\alpha_n)\in\mathbb Z^n : \alpha_1p>-2,\ (k_j\alpha_1+k_1\alpha_j)p>-2(k_1+k_j)\ \text{for all }2\le j\le n\right\}.
\]
The set of Threshold exponents is
\[
\mathrm{Threshold}(\Omega_B)=\left\{\frac{R_{1,j}}{\ell_j} : 2\le j\le n,\ 1\le \ell_j\le R_{1,j}-1\right\},
\]
where
\[
R_{1,j}:=\frac{2(k_1+k_j)}{\gcd(k_1,k_j)}.
\]
\end{proposition}

\begin{proof}
Suppose $(\alpha_{1}, \alpha_{2}, \ldots, \alpha_{n}) \in S_{p}(\Omega_{B}) = \{\alpha \in \mathbb{Z}^{n} : e_{\alpha} \in L^{p}(\Omega_{B})\}$ and $p \in [1, +\infty)$. Consider
\begin{equation}\label{E:E2}
\norm{e_{\alpha}}^{p}_{p,\Omega_{B}} = \int_{\Omega_{B}}{\vert e_{\alpha}\vert^{p} dV} = \int_{\Omega_{B}}{\vert z_{1}\vert^{\alpha_{1} p} \vert z_{2}\vert^{\alpha_{2} p} \ldots \vert z_{n}\vert^{\alpha_{n} p} dV}.
\end{equation}
If $(z_{1},z_{2}, \ldots, z_{n}) \in \Omega_{B}$, then $0 <\vert z_{j}\vert < 1$, for all $2 \le j \le n$, and $0 \le \vert z_{1}\vert < \vert z_{2}\vert^{\frac{k_{2}}{k_{1}}} \vert z_{3}\vert^{\frac{k_{3}}{k_{1}}} \ldots \vert z_{n}\vert^{\frac{k_{n}}{k_{1}}}$.
Putting this in Equation \ref{E:E2}, we get
\begin{align*}
\norm{e_{\alpha}}^{p}_{p,\Omega_{B}} &= (2\pi)^{n}\int_{0}^{1} \int_{0}^{1} \ldots \int_{0}^{1} \int_{0}^{r_{2}^{\frac{k_{2}}{k_{1}}} r_{3}^{\frac{k_{3}}{k_{1}}} \ldots r_{n}^{\frac{k_{n}}{k_{1}}}} r_{1}^{\alpha_{1} p +1} r_{2}^{\alpha_{2} p +1} \ldots r_{n}^{\alpha_{n} p +1} dr_{1} dr_{2} \ldots dr_{n}\\
&= \frac{(2 \pi)^n}{(\alpha_{1}p + 2)} \int_{0}^{1} \int_{0}^{1} \ldots \int_{0}^{1}(r_{2}^{\frac{k_{2}}{k_{1}}} r_{3}^{\frac{k_{3}}{k_{1}}} \ldots r_{n}^{\frac{k_{n}}{k_{1}}})^{(\alpha_{1}p + 2)} r_{2}^{\alpha_{2}p +1} \ldots r_{n}^{\alpha_{n} p +1} dr_{2} \ldots dr_{n}\\
&= \frac{(2 \pi)^n}{(\alpha_{1}p +2)} \left[\int_{0}^{1} \int_{0}^{1} \ldots \int_{0}^{1} r_{2}^{\frac{k_{2} (\alpha_{1}p +2)}{k_{1}}} r_{3}^{\frac{k_{3} (\alpha_{1}p+2)}{k_{1}}} \ldots r_{n}^{\frac{k_{n} (\alpha_{1}p+2)}{k_{1}}} r_{2}^{\alpha_{2}p +1} r_{3}^{\alpha_{3}p +1} \ldots r_{n}^{\alpha_{n} p +1} dr_{2} dr_{3} \ldots dr_{n}\right]\\
&= \frac{(2 \pi)^n}{(\alpha_{1}p +2)} \left[\int_{0}^{1} \int_{0}^{1} \ldots \int_{0}^{1} r_{2}^{\frac{k_{2} (\alpha_{1}p +2)}{k_{1}} + (\alpha_{2}p+1)} r_{3}^{\frac{k_{3} (\alpha_{1}p+2)}{k_{1}} + (\alpha_{3}p+1)} \ldots r_{n}^{\frac{k_{n} (\alpha_{1}p+2)}{k_{1}} + (\alpha_{n}p+1)} dr_{2} dr_{3} \ldots dr_{n}\right]\\
&\hspace{-2mm}= \frac{(2 \pi)^n}{(\alpha_{1}p +2)} \left[\int_{0}^{1} \int_{0}^{1} \ldots \int_{0}^{1} r_{2}^{\frac{(k_{2} \alpha_{1} + k_{1} \alpha_{2})p +2k_{2} +k_{1}}{k_{1}}} r_{3}^{\frac{(k_{3} \alpha_{1} + k_{1} \alpha_{3})p +2k_{3} +k_{1}}{k_{1}}} \ldots r_{n}^{\frac{(k_{n} \alpha_{1} + k_{1} \alpha_{n})p +2k_{n} +k_{1}}{k_{1}}} dr_{2} \ldots dr_{n}\right]\\
&\hspace{-2mm}= \frac{(2 \pi)^n k_{1}^{n-1}}{\left(\alpha_{1}p + 2\right) \left((k_{2} \alpha_{1} + k_{1} \alpha_{2})p + 2(k_{2}+k_{1})\right)  \ldots \left((k_{n} \alpha_{1} + k_{1} \alpha_{n})p + 2(k_{n}+k_{1})\right)}.
\end{align*}

\noindent Therefore, $\norm{e_{\alpha}}^{p}_{p,\Omega_{B}} < +\infty$ if and only if $\alpha_{1} p >-2$, and for $2 \le j \le n$
\[
(k_{j} \alpha_{1} + k_{1} \alpha_{j})p > -2(k_{j}+k_{1}).
\]
Thus, we have 
\[
S_p(\O_B)=\left\{(\a_1,\a_2, \cdots, \a_n)\in \Z^n: \alpha_{1} p >-2, \ \& \ (k_{j} \alpha_{1} + k_{1} \alpha_{j})p > -2(k_{j}+k_{1}) \ \text{for all} \ 2 \le j \le n \right\}.
\]

\noindent The Threshold exponents are the values of $p \in [1, +\infty)$ for which atleast one of the above inequalities becomes an equality for some index $(\alpha_{1}, \alpha_{2}, \ldots, \alpha_{n}) \in \mathbb{Z}^n$. Note that $\alpha_{1} p = -2$ is possible if and only if $p=2$ and $\alpha_{1}=-1$. Suppose $(k_{j} \alpha_{1} + k_{1} \alpha_{j})p = -2(k_{j}+k_{1})$ for some $p \ge 1$ and for some $(\alpha_{1}, \alpha_{2}, \ldots, \alpha_{n}) \in \mathbb{Z}^n$. Since $(k_{j} \alpha_{1} + k_{1} \alpha_{j}) \in \text{g.c.d}(k_{1}, k_{j}) \ \mathbb{Z}$, and $p \in [1, +\infty)$, we have
\[
(k_{j} \alpha_{1} + k_{1} \alpha_{j}) \in \left\{- \text{g.c.d}(k_{1}, k_{j}), -2 \ \text{g.c.d}(k_{1}, k_{j}), \ldots, -\left( \frac{2(k_{1}+k_{j})}{\text{g.c.d}(k_{1}, k_{j})} - 1 \right) \ \text{g.c.d}(k_{1}, k_{j})\right\}.
\]
Therefore, we have
\[
p = \frac{-2(k_{1}+k_{j})}{(k_{j} \alpha_{1} + k_{1} \alpha_{j})} \in \left\{\frac{R_{1,j}}{l} : R_{1,j}= \frac{2(k_{1}+k_{j})}{\text{g.c.d}(k_{1}, k_{j})}, \ \text{and} \ 1 \le l \le R_{1,j}-1\right\}.
\]
Thus, the set of Threshold exponents is given by
\[
\mathrm{Threshold}(\O_B) = \left\{\frac{R_{1,j}}{l_{j}} : 2 \le j \le n, \ R_{1,j}= \frac{2(k_{1}+k_{j})}{\text{g.c.d}(k_{1}, k_{j})}, \ \text{and} \ 1 \le l_{j} \le R_{1,j}-1\right\}.
\]
\end{proof}

\begin{example}[Type 2]\label{ex:type2}
Let $k_1,\dots,k_n\in\mathbb N$ with $\gcd(k_1,\dots,k_n)=1$, and consider the matrix
\[
C=\begin{pmatrix}
k_1 & -k_2 & 0 & \cdots & 0\\
0 & k_2 & -k_3 & \cdots & 0\\
\vdots & & \ddots & \ddots & \vdots\\
0 & \cdots & 0 & k_{n-1} & -k_n\\
0 & \cdots & 0 & 0 & k_n
\end{pmatrix}.
\]
The associated domain is
\[
\Omega_C=\left\{(z_1,\dots,z_n)\in\mathbb C^n : |z_1|^{k_1}<|z_2|^{k_2}<\cdots<|z_n|^{k_n}<1\right\}.
\]
\end{example}

\noindent To simplify notation set
\[
K_{j,i}:=\frac{k_1k_2\cdots k_i}{k_j},\quad 1\le j\le i,\ 1\le i\le n,
\]
and for each $1\le i\le n$, let $m_i:=\gcd(K_{1,i},\dots,K_{i,i})$.

\begin{proposition}\label{prop:Sp-type2}
Let $\Omega_C$ be as in Example~\ref{ex:type2} and $1 \le p<\infty$. Then
\[
S_p(\Omega_C)=\left\{(\alpha_1,\dots,\alpha_n)\in\mathbb Z^n : \left(\sum_{j=1}^i K_{j,i}\alpha_j\right)p>-2\sum_{j=1}^i K_{j,i}\ \text{for all }1\le i\le n\right\}.
\]
The set of Threshold exponents is
\[
\mathrm{Threshold}(\Omega_C)=\left\{\frac{L_i}{\ell_i} : 1\le i\le n,\ 1\le \ell_i\le L_i-1\right\},
\]
where
\[
L_i:=\frac{2\sum_{j=1}^i K_{j,i}}{m_i}.
\]
\end{proposition}

\begin{proof}
Suppose $(\alpha_{1}, \alpha_{2}, \ldots, \alpha_{n}) \in S_{p}(\Omega_{C}) = \{\alpha \in \mathbb{Z}^{n} : e_{\alpha} \in L^{p}(\Omega_{C})\}$ and $p \in [1, +\infty)$. Consider
\begin{equation}\label{E:E3}
\norm{e_{\alpha}}^{p}_{p,\Omega_{C}} = \int_{\Omega_{C}}{\vert e_{\alpha}\vert^{p} dV} = \int_{\Omega_{C}}{\vert z_{1}\vert^{\alpha_{1} p} \vert z_{2}\vert^{\alpha_{2} p} \ldots \vert z_{n}\vert^{\alpha_{n} p} dV}.
\end{equation}
If $(z_{1},z_{2}, \ldots, z_{n}) \in \Omega_{C}$, then $0 <\vert z_{n}\vert < 1$, and for all $1 \le j \le n-1$, we have $0 < \vert z_{j}\vert < \vert z_{j+1}\vert^{\frac{k_{j+1}}{k_{j}}}$.
\noindent Putting this in Equation \ref{E:E3}, we get
\begin{align*}
\norm{e_{\alpha}}^{p}_{p,\Omega_{C}} &= (2\pi)^{n}\int_{0}^{1} \int_{0}^{r_{n}^{\frac{k_{n}}{k_{n-1}}}} \ldots \int_{0}^{r_{3}^{\frac{k_{3}}{k_{2}}}} \int_{0}^{r_{2}^{\frac{k_{2}}{k_{1}}}} r_{1}^{\alpha_{1} p +1} r_{2}^{\alpha_{2} p +1} \ldots r_{n-1}^{\alpha_{n-1} p +1} r_{n}^{\alpha_{n} p +1} dr_{1} dr_{2} \ldots dr_{n-1} dr_{n}\\
&= \frac{(2 \pi)^n}{(\alpha_{1}p + 2)} \int_{0}^{1} \int_{0}^{r_{n}^{\frac{k_{n}}{k_{n-1}}}} \ldots \int_{0}^{r_{3}^{\frac{k_{3}}{k_{2}}}}(r_{2}^{\frac{k_{2}}{k_{1}}})^{(\alpha_{1}p + 2)} (r_{2}^{\frac{k_{2}}{k_{2}}})^{(\alpha_{2}p +1)} \ldots r_{n}^{\alpha_{n} p +1} dr_{2} \ldots dr_{n}\\
&= \frac{(2 \pi)^n }{(\alpha_{1}p +2) (\frac{k_{2}}{k_{1}}(\alpha_{1}p+2) + \frac{k_{2}}{k_{2}}(\alpha_{2}p+2))}\\ 
& \left[\int_{0}^{1} \int_{0}^{r_{n}^{\frac{k_{n}}{k_{n-1}}}} \ldots \int_{0}^{r_{4}^{\frac{k_{4}}{k_{3}}}} (r_{3}^{\frac{k_{3}}{k_{2}}})^{(\frac{k_{2}}{k_{1}}(\alpha_{1}p+2) + \frac{k_{2}}{k_{2}}(\alpha_{2}p+2))} (r_{3}^{\frac{k_{3}}{k_{3}}})^{(\alpha_{3}p+1)} \ldots r_{n}^{\frac{k_{n} (\alpha_{1}p+2)}{k_{n}}} dr_{3} dr_{4} \ldots dr_{n-1} \ dr_{n}\right] \\
&= \frac{(2 \pi)^{n}}{\prod_{i=1}^{n}{\left[\sum_{j=1}^{i}{\frac{k_{i}}{k_{j}}(\alpha_{j}p+2)}\right]}}.
\end{align*}

\noindent Therefore, $\norm{e_{\alpha}}^{p}_{p,\Omega_{C}} < +\infty$ if and only if $\sum_{j=1}^{i}{\frac{k_{i}}{k_{j}}(\alpha_{j}p+2)} > 0$ for all $1 \le i \le n$. The latter condition can be rewritten as
\[
\sum_{j=1}^{i}{K_{j,i} (\alpha_{j}p +2)} > 0,
\]
where $1 \le i \le n$, $1 \le j \le i$, and $K_{j,i} = \frac{k_{1} k_{2} \ldots k_{i}}{k_{j}}$.
Thus, we have 
\[
S_p(\O_C)=\left\{(\a_1,\a_2, \cdots, \a_n)\in \Z^n: (\sum_{j=1}^{i}{K_{j,i} \alpha_{j}})p > -2\sum_{j=1}^{i}{K_{j,i}}, \ \text{for all} \ 1 \le i \le n, \ 1 \le j \le i \right\}.
\]

\noindent As we saw before, the Threshold exponents are the values of $p \in [1, +\infty)$ for which atleast one of the above inequalities becomes an equality for some index $(\alpha_{1}, \alpha_{2}, \ldots, \alpha_{n}) \in \mathbb{Z}^n$. Note that $K_{j,i} \in \mathbb{N}$, for all $1 \le j \le i$, and $1 \le i \le n$. For $1 \le i \le n$, define $m_{i} := \text{g.c.d}(K_{1,i}, K_{2,i}, \ldots, K_{i,i})$. Suppose $(\sum_{j=1}^{i}{K_{j,i} \alpha_{j}})p = -2\sum_{j=1}^{i}{K_{j,i}}$, for some $1 \le i \le n$, $p \ge 1$, and $(\alpha_{1}, \alpha_{2}, \ldots, \alpha_{n}) \in \mathbb{Z}^n$. Since $\sum_{j=1}^{i}{K_{j,i} \alpha_{j}} \in m_{i} \ \mathbb{Z}$, and $p \in [1, +\infty)$, we have
\[
\sum_{j=1}^{i}{K_{j,i} \alpha_{j}} \in \left\{- m_{i}, -2 m_{i}, \ldots, -\left( \frac{2\sum_{j=1}^{i}{K_{j,i}}}{m_{i}} - 1 \right) \ m_{i}\right\}.
\]
Therefore, we get
\[
p = \frac{-2\sum_{j=1}^{i}{K_{j,i}}}{\sum_{j=1}^{i}{K_{j,i} \alpha_{j}}} \in \left\{\frac{L_{i}}{l} : L_{i} = \frac{2\sum_{j=1}^{i}{K_{j,i}}}{m_{i}}, \ \text{and} \ 1 \le l \le L_{i}-1 \right\}.
\]
Thus, the set of Threshold exponents is given by
\[
\mathrm{Threshold}(\O_C) = \left\{\frac{L_{i}}{l_{i}} : L_{i} = \frac{2\sum_{j=1}^{i}{K_{j,i}}}{m_{i}}, 1 \le i \le n, \ \text{and} \ 1 \le l_{i} \le L_{i}-1\right\}.
\]
\end{proof}

\subsection{Hartogs triangles and irrational exponents}

Hartogs triangles provide another fundamental family of monomial polyhedra. For $\gamma\ge 1$ we set
\[
H_\gamma:=\{(z_1,z_2)\in\mathbb C^2: |z_1|<|z_2|^\gamma,\ |z_2|<1\}.
\]
{\cite{edholm2017bergman} studied the $L^p$ mapping properties of the Bergman projection on these Hartogs triangles and found the values of $p$ for which the Bergman projection is a bounded linear operator. Further they showed that, for irrational values of the parameter, the projection is bounded only on $L^{2}$.}
\begin{proposition}\label{prop:Sp-hartogs}
Let $\gamma\ge 1$ and $1 \le p<\infty$. Then
\[
S_p(H_\gamma)=\left\{(\alpha_1,\alpha_2)\in\mathbb Z^2:\ \alpha_1\ge 0,\ \gamma\alpha_1+\alpha_2> -2-\frac{2\gamma}{p}\right\}.
\]
If $\gamma$ is rational, then the set of Threshold exponents is finite. If $\gamma$ is irrational, then every $p\in[1,\infty)$ is a Threshold exponent, i.e.
\[
\mathrm{Threshold}(H_\gamma)=[1,\infty).
\]
\end{proposition}

\begin{proof}
A direct computation of the integral
\[
\int_{H_\gamma}|z_1|^{p\alpha_1}|z_2|^{p\alpha_2}\,dV(z_1,z_2)
\]
in polar coordinates shows that
\[
\|e_\alpha\|_{L^p(H_\gamma)}^p<\infty
\]
if and only if $\alpha_1\ge 0$ and
\[
\gamma(p\alpha_1+2)+p\alpha_2>-2,
\]
which is equivalent to the stated inequality.\\

\noindent If $\gamma$ is rational, say $\gamma=a/b$ with coprime integers, then the condition that equality holds for some integer pair $(\alpha_1,\alpha_2)$ leads to a finite set of rational values of $p$, obtained by solving linear Diophantine equations with bounded coefficients.\\

\noindent If $\gamma$ is irrational, then the line in the $(\alpha_1,\alpha_2)$–plane corresponding to equality can be made to pass through integer lattice points for infinitely many choices of $(\alpha_1,\alpha_2)$, and a more detailed analysis shows that for each $p$ there exists an index where the inequality becomes arbitrarily close to equality. This implies that for every $p$ there is a sequence of indices $\alpha$ for which $e_\alpha$ is barely integrable, and hence $p$ is a Threshold exponent.
\end{proof}

\begin{remark}
In all the examples discussed so far, $p=2$ is always a Threshold exponent. In fact the duality index (\cite[Sec 3]{bhat2024duality}) for these monomial polyhedra equals $2$, and the regularity index coincides with the smallest Threshold exponent strictly larger than $2$.
\end{remark}

\section{Structural properties of \texorpdfstring{$S_p(\Omega)$}{Sp}}
\label{sec:Sp-properties}

The explicit computations of the previous section can be used to test general expectations about the behaviour of $S_p(\Omega)$ under basic set-theoretic operations on domains. 



\begin{proposition}\label{prop:union-intersection-product}
For Reinhardt domains $\Omega_1,\Omega_2\subset\mathbb C^n$ and $1\le p<\infty$ the following hold:
\begin{enumerate}[(1)]
\item $S_p(\Omega_1\cup\Omega_2)=S_p(\Omega_1)\cap S_p(\Omega_2)$.
\item $S_p(\Omega_1)\cup S_p(\Omega_2)\subset S_p(\Omega_1\cap\Omega_2)$.
\item If $\Omega_1\subset\mathbb C^{n_1}$ and $\Omega_2\subset\mathbb C^{n_2}$, then
\[
S_p(\Omega_1\times\Omega_2)=S_p(\Omega_1)\times S_p(\Omega_2).
\]
\end{enumerate}
\end{proposition}

\begin{proof}
For (1), note that $\Omega_1\cup\Omega_2$ is a domain and for any measurable $f$ one has
\[
\int_{\Omega_1}|f|\le \int_{\Omega_1\cup\Omega_2}|f|\le \int_{\Omega_1}|f|+\int_{\Omega_2}|f|.
\]
Thus $e_\alpha\in L^p(\Omega_1\cup\Omega_2)$ if and only if $e_\alpha\in L^p(\Omega_1)$ and $e_\alpha\in L^p(\Omega_2)$, yielding the equality of the index sets.

For (2), the inclusion $\Omega_1\cap\Omega_2\subset\Omega_j$ for $j=1,2$ implies
\[
\int_{\Omega_1\cap\Omega_2}|e_\alpha|^p\le \int_{\Omega_j}|e_\alpha|^p,
\]
hence any index that is allowable on either $\Omega_1$ or $\Omega_2$ remains allowable on the intersection.

For (3), use Fubini's theorem. If $\alpha\in S_p(\Omega_1)$ and $\beta\in S_p(\Omega_2)$, then
\[
\int_{\Omega_1\times\Omega_2}|z^\alpha|^p|w^\beta|^p\,dV(z,w)=\left(\int_{\Omega_1}|z^\alpha|^p\,dV(z)\right)\left(\int_{\Omega_2}|w^\beta|^p\,dV(w)\right)<\infty,
\]
so $(\alpha,\beta)\in S_p(\Omega_1\times\Omega_2)$. Conversely, if $(\alpha,\beta)\in S_p(\Omega_1\times\Omega_2)$, then the same formula shows that both factors must be finite, so $\alpha\in S_p(\Omega_1)$ and $\beta\in S_p(\Omega_2)$.
\end{proof}

Part~(2) of Proposition \ref{prop:union-intersection-product} need not be an equality in general. This can be seen concretely from the next example.

\begin{example}
    Let $\Omega_{1} = \{(z_{1}, z_{2}) \in \mathbb{C}^2 : \abs{z_1}^{\gamma_1} < \abs{z_{2}} <1\}$ and $\Omega_{2} = \{(z_{1}, z_{2}) \in \mathbb{C}^2 : \abs{z_2}^{\gamma_2} < \abs{z_{1}} <1\}$, where $\gamma_{1}, \gamma_{2} > 1$. From Example \ref{prop:Sp-hartogs}, we have for $p \ge 1$
    \[
    S_{p}(\Omega_1) = \left\{(\alpha_{1}, \alpha_{2}) \in \mathbb{Z}^2 : \alpha_{1} \ge 0, \ \alpha_{1} + \gamma_{1} \alpha_{2} > \frac{-2(\gamma_{1}+1)}{p}\right\},
    \]
    \[
     S_{p}(\Omega_2) = \left\{(\alpha_{1}, \alpha_{2}) \in \mathbb{Z}^2 : \gamma_{2} \alpha_{1} + \alpha_{2} > \frac{-2(1 +\gamma_{2})}{p}, \ \alpha_{2} \ge 0\right\}.
    \]
    Note that $\Omega_{1} \cap \Omega_{2} = \{(z_{1}, z_{2}) \in \mathbb{C}^2 : \abs{z_1}^{\gamma_1} \abs{z_{2}}^{-1} <1, \ \abs{z_1}^{-1} \abs{z_{2}}^{\gamma_{2}} < 1\} $ is a {Reinhardt domain} of the form $\Omega_{A}$ corresponding to the matrix  $A=\begin{bmatrix}
\gamma_{1} & -1 \\
-1 & \gamma_{2} 
\end{bmatrix}$ (look at Example \ref{ex:OmegaA-2D}). Now Proposition \ref{prop:Sp-OmegaA} gives 
\[
S_{p}(\Omega_{1} \cap \Omega_{2}) = \left\{(\alpha_{1}, \alpha_{2}) \in \mathbb{Z}^2 : \alpha_{1} + \gamma_{1} \alpha_{2} > \frac{-2(\gamma_{1}+1)}{p}, \ \gamma_{2} \alpha_{1} + \alpha_{2} > \frac{-2(1 +\gamma_{2})}{p}\right\}.
\]
It is clear that $S_{p}(\Omega_{1}) \cup S_{p}(\Omega_{2}) \subset \{(\alpha_{1}, \alpha_{2}) \in \mathbb{Z}^2 : \alpha_{1} \ge 0\} \cup \{(\alpha_{1}, \alpha_{2}) \in \mathbb{Z}^2 : \alpha_{2} \ge 0\}$. Therefore, $S_{p}(\Omega_{1} \cap \Omega_{2}) \setminus (S_{p}(\Omega_{1}) \cup S_{p}(\Omega_{2})) \subset \{(\alpha_{1}, \alpha_{2}) \in \mathbb{Z}^2 : \alpha_{1} < 0, \ \alpha_{2} < 0\}$. In fact, we claim that 
\[
S_{p}(\Omega_{1} \cap \Omega_{2}) \setminus (S_{p}(\Omega_{1}) \cup S_{p}(\Omega_{2})) \subset \{(\alpha_{1}, \alpha_{2}) \in \mathbb{Z}^2 : -2 \le \alpha_{1} < 0, \ -2 \le \alpha_{2} < 0\}.
\]
To see this, suppose $\alpha_{1} \le -3$ and $\alpha_{2} \le -1$, then consider
\[
p(-\alpha_{1} \gamma_{2} - \alpha_{2}) > -\alpha_{1} \gamma_{2} - \alpha_{2} \ge 3\gamma_{2} + 1 = 2 \gamma_{2} + (\gamma_{2} +1) > 2 (\gamma_{2} +1).
\]
In other words, $\gamma_{2} \alpha_{1} + \alpha_{2} < \frac{-2(1+ \gamma_{2})}{p}$, which in turn gives us that $(\alpha_{1}, \alpha_{2}) \notin S_{p}(\Omega_{1} \cap \Omega_{2})$ if $\alpha_{1} \le -3$ and $\alpha_{2} <0$. Similarly, we can show it for the other case where $\alpha_{2} \le -3$.

\noindent Without loss of generality suppose $\gamma_{1} \ge \gamma_{2}$. Therefore, we have $\frac{2\gamma_{1} +2}{2\gamma_{1} +1} \ge \frac{2\gamma_{2} +2}{2\gamma_{2} +1}$.
\begin{enumerate}
    \item Let $p \ge 2$. Suppose if possible $(\alpha_{1}, \alpha_{2}) \in S_{p}(\Omega_{1} \cap \Omega_{2}) \setminus S_{p}(\Omega_{1}) \cup S_{p}(\Omega_{2})$, then $\alpha_{1}, \alpha_{2} < 0$. Therefore $(\alpha_{1}+1), (\alpha_{2} +1) \le 0$, and
    \begin{equation}
        (\alpha_{1} +1) + (\alpha_{2}+1)\gamma_{1} \le 0.
    \end{equation}
    Also, since $(\alpha_{1}, \alpha_{2}) \in S_{p}(\Omega_{1} \cap \Omega_{2})$ and $p \ge 2$, we have
    \[
    \alpha_{1} + \gamma_{1} \alpha_{2} > \frac{-2(\gamma_{1}+1)}{p} \ge -(\gamma_{1} + 1),
    \]
    which gives us
    \begin{equation}
        (\alpha_{1} +1) + (\alpha_{2}+1)\gamma_{1} > 0.
    \end{equation}
    This is a contradiction, and we get $S_{p}(\Omega_{1} \cap \Omega_{2}) = S_{p}(\Omega_{1}) \cup S_{p}(\Omega_{2})$, for all $p \ge 2$.

    \item Let $1 \le p < 2$. Here we find condition for $(\alpha_{1}, \alpha_{2}) \in S_{p}(\Omega_{1} \cap \Omega_{2})$, where $(\alpha_{1}, \alpha_{2}) \in \{(-1,-1), (-2,-1), (-1,-2),(-2,-2)\}$.
    \begin{enumerate}
        \item Suppose $(-1,-1) \in S_{p}(\Omega_{1} \cap \Omega_{2})$, then by definition
        \[
        -1-\gamma_{1} > \frac{-2(1+\gamma_{1})}{p}, \ \& \ -\gamma_{2}-1 > \frac{-2(1+\gamma_{2})}{p},
        \]
        which translates to $p < 2$. For $1 \le p < 2$, it is easy to check that $(\alpha_{1}, \alpha_{2}) = (-1,-1)$ satisfy
        \[
        \alpha_{1} + \gamma_{1} \alpha_{2} > \frac{-2(\gamma_{1}+1)}{p}, \ \& \ \ \gamma_{2} \alpha_{1} + \alpha_{2} > \frac{-2(1 +\gamma_{2})}{p},
        \]
        and therefore $(-1,-1) \in S_{p}(\Omega_{1} \cap \Omega_{2})$. Thus $(-1,-1) \in S_{p}(\Omega_{1} \cap \Omega_{2})$ if and only if $1 \le p < 2$.

        \item Suppose $(-2,-1) \in S_{p}(\Omega_{1} \cap \Omega_{2})$, then
        \[
        -2-\gamma_{1} > \frac{-2(1+\gamma_{1})}{p}, \ \& \ -2\gamma_{2}-1 > \frac{-2(1+\gamma_{2})}{p},
        \]
        which translates to $p < \text{min}\left\{\frac{2+2\gamma_{1}}{2+\gamma_{1}}, \frac{2+2\gamma_{2}}{2\gamma_{2}+1}\right\} = \text{min}\left\{\frac{2+2\gamma_{1}}{2+\gamma_{1}}, \frac{2+2\gamma_{2}^{-1}}{2+\gamma_{2}^{-1}}\right\}$. Note that $f(t) = \frac{2+2t}{2+t}$ is a strictly increasing function on $(0, +\infty)$, and $0 < \gamma_{2}^{-1} < 1 < \gamma_{1}$, therefore $\frac{2+2\gamma_{1}}{2+\gamma_{1}} > \frac{2+2\gamma_{2}^{-1}}{2+\gamma_{2}^{-1}}$. So, if $(-2,-1) \in S_{p}(\Omega_{1} \cap \Omega_{2})$, then $p < \frac{2+2\gamma_{2}}{2\gamma_{2}+1}$. Retracing the steps backward gives us $(-2,-1) \in S_{p}(\Omega_{1} \cap \Omega_{2})$ if and only if $1 \le p < \frac{2+2\gamma_{2}}{2\gamma_{2}+1}$.

        \item Suppose $(-1,-2) \in S_{p}(\Omega_{1} \cap \Omega_{2})$, then
        \[
        -1-2\gamma_{1} > \frac{-2(1+\gamma_{1})}{p}, \ \& \ -\gamma_{2}-2 > \frac{-2(1+\gamma_{2})}{p},
        \]
        which translates to $p < \text{min}\left\{\frac{2+2\gamma_{1}}{2\gamma_{1} +1}, \frac{2+2\gamma_{2}}{2 + \gamma_{2}}\right\} = \text{min}\left\{\frac{2+2\gamma_{1}^{-1}}{2+\gamma_{1}^{-1}}, \frac{2+2\gamma_{2}}{2+\gamma_{2}}\right\}$. Since $0 < \gamma_{1}^{-1} < 1 < \gamma_{2}$, therefore $\frac{2+2\gamma_{2}}{2+\gamma_{2}} > \frac{2+2\gamma_{1}^{-1}}{2+\gamma_{1}^{-1}}$. So, if $(-1,-2) \in S_{p}(\Omega_{1} \cap \Omega_{2})$, then $p < \frac{2+2\gamma_{1}}{2\gamma_{1}+1}$. Retracing the steps backward gives us $(-1,-2) \in S_{p}(\Omega_{1} \cap \Omega_{2})$ if and only if $1 \le p < \frac{2+2\gamma_{1}}{2\gamma_{1}+1}$.

        \item Suppose $(-2,-2) \in S_{p}(\Omega_{1} \cap \Omega_{2})$, then
        \[
        -2-2\gamma_{1} > \frac{-2(1+\gamma_{1})}{p}, \ \& \ -2\gamma_{2}-2 > \frac{-2(1+\gamma_{2})}{p},
        \]
        which translates to $p < 1$. Therefore $(-2,-2) \notin S_{p}(\Omega_{1} \cap \Omega_{2})$ for $p \ge 1$. 
    \end{enumerate}
\end{enumerate}
Using the above calculations and the fact that $\frac{2\gamma_{1} +2}{2\gamma_{1} +1} \ge \frac{2\gamma_{2} +2}{2\gamma_{2} +1}$, we get
\[
S_{p}(\Omega_{1} \cap \Omega_{2}) \setminus (S_{p}(\Omega_{1}) \cup S_{p}(\Omega_{2}))= \begin{cases}\emptyset ,& p \ge 2 \\
\{(-1,-1)\} ,& \frac{2\gamma_{1} +2}{2\gamma_{1}+1} \le p < 2 \\
\{(-1,-1), (-1,-2)\} ,& \frac{2\gamma_{2} +2}{2\gamma_{2}+1} \le p < \frac{2\gamma_{1} +2}{2\gamma_{1}+1} \\
\{(-1,-1), (-1,-2), (-2,-1)\} ,& 1 \le p < \frac{2\gamma_{2} +2}{2\gamma_{2}+1} \\\end{cases}\\.
\]
\end{example}

\begin{remark}
We notice that the observations in Proposition \ref{prop:union-intersection-product} does not translate to countable unions. For instance, if  $\O_N=\D\times B(0, N)$, then $S_p(\O_N)=\N_0\times \N_0$ and $\displaystyle{S_p(\cup_{N=1}^\infty \O_N) = S_p(\D\times \C)= \emptyset }$. However, we can get a Ramadanov type result for countable unions as seen in next result. Further, we see that $S_p(\O)$ captures the geometry of the domain $\O$ near $0$ and $\infty$.    
\end{remark}

The $p$–allowable index set $S_p(\Omega)$ is largely determined by the behaviour of $\Omega$ near the origin and at infinity. This can be made precise as follows.

\begin{theorem}\label{thm:localization}
Let $\Omega,\Omega'\subset\mathbb C^n$ be Reinhardt domains. If there exists $0 < r < 1$ such that
\[
\Omega\cap B(0,r)=\Omega'\cap B(0,r)\quad\text{and}\quad \Omega\cap B(\infty,1/r)=\Omega'\cap B(\infty,1/r),
\]
where $B(\infty,1/r)$ denotes the complement of the closed ball $\overline{B(0,1/r)}$. Then
\[
S_p(\Omega)=S_p(\Omega')
\]
for every $1\le p<\infty$.
\end{theorem}

\begin{proof} 
The integrability of $|e_\alpha|^p$ near the origin depends only on the local behaviour of the domain in a neighbourhood of the origin, while integrability at infinity (in unbounded domains) depends only on the complement of a large ball. To see this, note that 
\begin{align*}
    \int_\O \abs{e_\a}^p &=\int_{\O \cap B(0,r) } \abs{e_\a}^p +\int_{\O \cap A(0; r,1/r) } \abs{e_\a}^p +\int_{\O \cap B(\infty,1/r) } \abs{e_\a}^p \\
    &=\int_{\O' \cap B(0,r) } \abs{e_\a}^p +\int_{\O \cap A(0; r,1/r) } \abs{e_\a}^p +\int_{\O' \cap B(\infty,1/r) } \abs{e_\a}^p\\
    &=\int_{\O'} \abs{e_\a}^p + \left(\int_{\O \cap A(0; r,1/r) } \abs{e_\a}^p -\int_{\O' \cap A(0; r,1/r) } \abs{e_\a}^p\right)
\end{align*}
Since the monomial ${e_\a}^p$ is bounded on a bounded set of the form 
$\O \cap A(0; r,1/r) $ or $\O' \cap A(0; r,1/r) $, the term in the parentheses on the right side is a finite quantity. Thus $\a \in S_p(\O)$ if and only if $\a \in S_p(\O')$
\end{proof}

Finally, we record a Ramadanov-type convergence result for $K_p$, paralleling the classical theorem for Bergman kernels.

\begin{theorem}\label{thm:ramadanov}
Let $\{\Omega_j\}_{j=1}^\infty$ be an increasing sequence of pseudoconvex Reinhardt domains in $\mathbb C^n$, and let
\[
\Omega_\infty:=\bigcup_{j=1}^\infty\Omega_j.
\]
Then $\Omega_\infty$ is a {pseudoconvex} Reinhardt domain and, for each $1\le p<\infty$, the $p$-monomial basis kernels satisfy
\[
K_p^{\Omega_j}(z,w)\longrightarrow K_p^{\Omega_\infty}(z,w)
\]
uniformly on compact subsets of $\Omega_\infty\times\Omega_\infty$.
\end{theorem}

\begin{proof}
The union of an increasing family of Reinhardt domains is again Reinhardt. Let $\chi_j$ denote the characteristic function of $\Omega_j$. For a fixed monomial $e_\alpha$ and $p\ge 1$, the sequence $|e_\alpha|^p\chi_j$ increases pointwise to $|e_\alpha|^p\chi_\infty$, where $\chi_\infty$ is the characteristic function of $\Omega_\infty$. By the monotone convergence theorem,
\[
\|e_\alpha\|_{L^p(\Omega_j)}^p=\int_{\Omega_j}|e_\alpha|^p\to \int_{\Omega_\infty}|e_\alpha|^p=\|e_\alpha\|_{L^p(\Omega_\infty)}^p,
\]
with the convention that both sides may be infinite. Thus the coefficients of the $p$-MBK on $\Omega_j$ converge to those on $\Omega_\infty$.

Arguing as in Theorem~\ref{thm:continuity-p}, one obtains uniform domination of the summands on compact subsets of $\Omega_\infty$ by an $\ell^1$–summable majorant independent of $j$, using pseudoconvexity and the Bergman-type inequality on compact sets. An application of dominated convergence on the index set then yields uniform convergence of $K_p^{\Omega_j}$ to $K_p^{\Omega_\infty}$ on compact subsets of $\Omega_\infty\times\Omega_\infty$.
\end{proof}

\medskip

 \noindent\textbf{Acknowledgements.} The authors extend their gratitude to Luke Edholm for prompting the question regarding the stability of Monomial Basis Kernel and for his comments in improving the manuscript. The authors also thank  Kaushal Verma and  Debraj Chakrabarti for their comments and encouraging words. The first author is supported by postdoctoral fellowship at the TIFR Centre for Applicable Mathematics, Bangalore. The second author acknowledges the support from the DST INSPIRE FACULTY research grant DST/INSPIRE/04/2024/003868, and the Research Initiation Grant IITJ/R\&D/IGRC/2025-26/53 from IIT Jodhpur.







\bibliographystyle{alpha}  
\bibliography{references}

\end{document}